\documentclass{amsart}
\usepackage{amssymb}
\usepackage{latexsym}
\usepackage{graphicx}
\usepackage{hyperref}
\newcommand{\Z}{{\mathbb Z}}
\newcommand{\R}{{\mathbb R}}

\newcommand{\EE}{{\mathbb E}}

\newtheorem{lemma}{Lemma}[section]
\newtheorem{theorem}[lemma]{Theorem}

\newtheorem{corollary}[lemma]{Corollary}
\newtheorem{definition}[lemma]{Definition}

\newcommand{\nn}{\nonumber}
\newcommand{\be}{\begin{equation}}
\newcommand{\ee}{\end{equation}}

\newcommand{\ol}{\overline}

\newcommand{\spr}[2]{\langle #1 , #2 \rangle}

\newcommand{\E}{\mathrm{e}}
\newcommand{\I}{\mathrm{i}}

\DeclareMathOperator{\dist}{dist}

\numberwithin{equation}{section}

\begin{document}

\title{Exponential dynamical localization for the almost Mathieu operator}

\author[S.\ Jitomirskaya]{Svetlana Jitomirskaya}
\address{University of California, Irvine, CA 92717}
\email{\href{szhitomi@math.uci.edu}{szhitomi@math.uci.edu}}

\author[H.\ Kr\"uger]{Helge Kr\"uger}
\address{Mathematics 253-37, Caltech, Pasadena, CA 91125}
\email{\href{helge@caltech.edu}{helge@caltech.edu}}
\urladdr{\href{http://www.its.caltech.edu/~helge/}{http://www.its.caltech.edu/~helge/}}

\thanks{H.\ K.\ was supported by a fellowship of the Simons foundation.}

\date{\today}

\keywords{Almost Mathieu operator, dynamical localization}
\subjclass[2000]{Primary 81Q10; Secondary 37D25, 47B36}

\begin{abstract}
 We prove that the exponential moments of the position
 operator stay bounded for the supercritical almost Mathieu operator
 with Diophantine frequency.
\end{abstract}

\maketitle

\section{Introduction}

As pointed out in \cite{rjls95}, the spectral property of Anderson
localization of a self-adjoint operator $H$
does not let us conclude much about the behavior of the
time evolution $\E^{-\I t H}$ \footnote{If $\psi(t)$ solves $\I\dot\psi = H \psi$,
then $\psi(t) = \E^{-\I t H}\psi(0)$.}, with dynamical  localization-type
conclusions requiring either additional or separate
arguments. Various dynamical localization formulations have been
suggested and proved, by different methods, for most popular models,
particularly, for the Anderson model. Aside
from uniform dynamical localization that happens rarely for ergodic families
\cite{rjls95,j96}, the strongest dynamical localization property is
the exponential (in space) rate of decay of the expected overlap,
that is 
\be \label{edl}
   \EE\sup_{t\in\R} |\spr{\delta_k}{\E^{- \I t H} \delta_{\ell}}|
    \leq C \E^{-\gamma |k - \ell|}
  \ee
A stronger, but equivalent in all known examples, formulation is
(\ref{edl}) with $\E^{-\I t H}$ replaced with an arbitrary bounded
  function $f(H),$ or, equivalently, by the exponential decay of
$\EE\sum_s ( |\varphi_s(0)||\varphi_s (k)|),$ where $\{\varphi_s\}_s$
is a complete set of orthonormalized eigenfunctions (and the sum may
be localized in energy, if needed). As pointed out in \cite{ag}, this
leads to various interesting physical conclusions, for example the
exponential decay of the two-point function at the ground state and
positive temperatures with correlation length staying uniformly
bounded as temperature goes to zero. Moreover, it is this conclusion
that is often implicitly assumed as manifesting localization, in
physics literature. It is therefore desirable to establish it for
physically relevant models.

It is natural in this context to define the {\it exponential decay rate in expectation} (obviously
connected to the minimal inverse correlation length) as

\be\label{gamma}
 \gamma := \liminf_{k\to\infty}\left( -\frac{\ln \EE (\sum_{s} |\varphi_s(0)|\cdot |\varphi_s (k)|)}{|k|}\right)
\ee
with a relevant question being whether or not it is positive. This
definition can be localized  to an energy range by summing over the
eigenfunctions with energies falling in the range, in which case it is
linked to the minimal inverse correlation length for Fermi energies
falling in that range.

The corresponding question for the {\em Anderson model},
i.e. for the potential being independent identically
distributed random variables, was solved in \cite{dks83, ks80}
for the one-dimensional case and in \cite{a94, asfh01}
for higher dimensions throughout the regimes where corresponding
proofs of localization work, thus establishing positivity of $\gamma.$. The corresponding
result for continuum operators was proven in \cite{aenss06}. 

In this paper we establish the first result of this type for a
non-random ergodic operator.\footnote{ aside from families with
  uniform localization, where this conclusion holds trivially. It
  should be noted that the existing examples with uniform localization 
  (e.g. \cite{dg}) are of a rather artificial nature and do not
correspond to physical systems}
The most natural and popular operator
that fits this description is 
the almost Mathieu operator given by
\be\label{am}\begin{split}
 H_{\lambda,\alpha,\theta} : \ell^2(\Z) &\to\ell^2(\Z),\\
 H_{\lambda,\alpha,\theta} u(n) &= u(n+1) + u(n-1) + 2 \lambda \cos(2\pi(n\alpha + \theta)) u(n),
\end{split}\ee
where $\lambda > 0$, $\alpha$ is irrational, $\theta \in \R$,
and $\ell^2(\Z)$ denotes the space of square summable bi-infinite
sequences. We will be interested in the supercritical regime
that is $\lambda > 1,$ which is the regime of localization for
almost every value of $(\alpha,\theta)$. 

The frequency $\alpha$ is Diophantine
if there exist $\kappa, \tau > 0$ such that
\be
 \|q\alpha\| = \dist(q\alpha, \Z) \geq \frac{\kappa}{q^{\tau}}
\ee
for all integers $q \geq 1$. We will also use the notation
$0\leq \beta(\alpha) := \limsup -\frac{\ln  \|q\alpha\|}{q},$ so
$\beta(\alpha)=0$ for Diophantine $\alpha.$ Frequencies $\alpha$ with
small $\beta(\alpha)$ (with the quantifier depending on the context)
are often called weakly Liouville.  It was shown in \cite{j99} that
for $\lambda > 1$ and $\alpha$ Diophantine {\em Anderson localization}
holds, that is for almost every $\theta \in \R$, the operator
$H_{\lambda,\alpha,\theta}$ has pure point spectrum with 
exponentially decaying eigenfunctions. This has been extended to the 
weakly Liouville case in \cite{10M}.

As far as quantum dynamics goes,  the following
dynamical statements are known:
\begin{enumerate}
 \item Jitomirskaya and Last have shown in \cite{jl98} that for $\lambda > \frac{15}{2}$,
  $\alpha$ Diophantine, and almost every $\theta$ we have
  \be
   \sup_{t\in\R} |\spr{\delta_k}{\E^{-\I t H_{\lambda,\alpha,\theta}} \delta_{\ell}}|
    \leq C \E^{-\gamma |k - \ell|}
  \ee
  where $\gamma > 0$ is independent of $\theta$, $k$, and $\ell$,
  but $C > 0$ is allowed to depend on $\theta$ and $\ell$, but
  is independent of $k$. 
 \item Germinet and Jitomirskaya have shown in \cite{gj01} that
  for $\lambda > 1$ and $\alpha$ Diophantine, we have that for
  every $q > 1$
  \be
   \int_{0}^{1} \sup_{t} \sum_{n\in\Z} (1 + |n|)^{q}
   |\spr{\delta_n}{\E^{-\I t H_{\lambda,\alpha,\theta}} \delta_0}| d\theta
    <\infty.
  \ee
  This property is called {\em strong dynamical localization} in
  \cite{gj01} and a number of publications on the Anderson model and
  quasiperiodic operators.
  As the conclusion of Theorem~\ref{thm:main} is stronger a better name
  might be {\em polynomial dynamical localization in expectation}.
 \item It is known since \cite{g,as} that a Diophantine-type assumption
   is necessary even for spectral localization.
 \item Bourgain and Jitomirskaya \cite{bj00} have proven
  dynamical localization results of the type (ii) for more general quasi-periodic
  models.
\end{enumerate}

In this note we prove the above described exponential rate of decay for
supercritical  almost Mathieu operator with Diophantine (and even
weakly Liouville) frequencies:

\begin{theorem}\label{thm:main}
 Let $\lambda > 1.$ Then there exists $\beta >0$ such that for $\alpha$
 with $\beta(\alpha)<\beta$ we have $\gamma>0$ with $\gamma$ defined
 in (\ref{gamma}). 
\end{theorem}
This immediately implies
\begin{corollary}\label{cor:main}
Under the conditions of Theorem \ref{thm:main} there exists $\gamma > 0$
 and $C > 0$ such that
 \be
   \int_{0}^{1} \sup_{t\in\R} |\spr{\delta_{\ell}}{\E^{-\I t H_{\lambda,\alpha,\theta}} \delta_{k}}| d\theta
    \leq C \E^{-\gamma |k-\ell|}
 \ee
 for all $k,\ell \in\Z$.
\end{corollary}

The conclusion of this corollary can be best called
{\em exponential dynamical localization in expectation}.
It is clear that this result implies both the results of 
\cite{gj01} and \cite{jl98}. 
The proof of Theorem~\ref{thm:main} is based on the results
of \cite{aj10}. As the bounds on eigenfunction decay
obtained in \cite{aj10} are nonquantitative also our
bounds on $\gamma$ are. 

It is an interesting and natural question if the limit in the
right-hand side of
(\ref{gamma}) exists in general and if for one dimensional operators $\gamma $
can be equal to the minimal Lyapunov exponent as that is how the eigenfunctions often decay
(for the Diophantine almost Mathieu case the almost Lyapunov decay of
the eigenfunctions was shown in \cite{j99}). We plan to investigate
this question in a future work. It is also interesting to establish
positive decay rate for other well-studied models with localization,
for example, for other quasiperiodic operators in the regime of
positive Lyapunov exponents.

\bigskip

The proof of Theorem~\ref{thm:main} splits into two parts.
In Section~\ref{sec:redef}, we reduce the question to a question
of decay estimates of eigenfunctions with a given localization
center. As we hope that the results will be useful in other
contexts, we formulate this section for general families of
operators acting on $\ell^2(\Z^d)$.
In Section~\ref{sec:al}, we deduce the appropriate
decay estimate from the results of \cite{aj10}.

\section{Reduction to a question about eigenfunctions}
\label{sec:redef}

As the results in the following are valid for general families
of operators, we state them in this generality. Let $\{H_x\}_{x\in X}$
be a family of self-adjoint operators on $\ell^2(\Z^d),$ and $(X,\mu)$ a probability
space. Assume that for $\mu$ almost every $x\in X$ the spectrum
of $H_x$ is pure point. Denote by $\varphi_{x;s}$ an orthonormal
basis of $\ell^2(\Z^d)$ consisting of eigenfunctions of $H_x$.
For each $\varphi_{x; s}$, we let $n_{x;s} \in\Z^d$ be  such that
\be
 |\varphi_{x;s}(n_{x;s})| = \| \varphi_{x;s}\|_{\ell^{\infty}(\Z^d)}.
\ee
It is shown for $d=1$ in Section~2 of \cite{gjls97}. 
that $\varphi_{x;s}$ and $n_{x;s}$ can be chosen
to be measurable functions of $x$. The general case
can be found in \cite{gk99}.

Fixing $\lambda > 1$ and $\alpha$ with $\beta(\alpha)$ sufficiently small,
the results of \cite{j99} imply that the family of operators
$\{H_{\lambda,\alpha,\theta}\}_{\theta\in [0,1]}$ satisfy the
above properties when $[0,1]$ is equipped with the normalized
Lebesgue measure.

We state the next theorem in a form more general than is
necessary for application to the almost Mathieu operator. It
is our hope that it will be useful in other contexts. The
main difference to a more naive estimate is that one
can exploit the orthogonality of the $\varphi_{x;s}$ when
estimating $\sum_{n_{x;s} = n} |\varphi_{x;s}(\ell)|^2$.

\begin{theorem}\label{thm:A}
 We have for almost every $x$ that
 \be\begin{split}
 |\spr{\delta_k}{\E^{-\I t H_x} \delta_{\ell}}| 
&\leq \sum_{s} |\ol{\varphi_{x;s}(k)} \varphi_{x;s}(\ell)| \\
\nn&\leq \sum_{n}   \left( \sum_{n_{x;s} = n} |\ol{\varphi_{x;s}(k)}|^2 
    \sum_{n_{x;s} = n} |\varphi_{x;s}(\ell)|^2\right)^{\frac{1}{2}}
 \end{split}\ee
 and 
 \be\begin{split}
  \int |\spr{\delta_k}{\E^{-\I t H_x} \delta_{\ell}}| d\mu(x)&
   \leq \int \sum_{s} |\ol{\varphi_{x;s}(k)} \varphi_{x;s}(\ell)| d\mu(x) \\
\nn  & \leq \sum_{n} \left(\int \sum_{n_{x;s} = n} |\varphi_{x;s}(k)|^2 d\mu(x)
    \int \sum_{n_{x;s} = n} |\varphi_{x;s}(\ell)|^2 d\mu(x)\right)^{\frac{1}{2}}.
 \end{split} \ee
\end{theorem}

\begin{proof}
 Denote by $E_{x;s}$ the eigenvalues of $H_x$ such that
 $H_x \varphi_{x;s} = E_{x;s} \varphi_{x;s}$. We have that
 \[
  \spr{\delta_k}{\E^{-\I t H_x} \delta_{\ell}} = 
   \sum_{s} \ol{\varphi_{x;s}(k)} \varphi_{x;s}(\ell) \E^{-i t E_{x;s}}.
 \]
 By the triangle inequality and reordering the sum,
 we obtain
 \[
  |\spr{\delta_k}{\E^{-\I t H_x} \delta_{\ell}}|
    \leq \sum_{n\in\Z} \sum_{n_{x;s} = n} |\ol{\varphi_{x;s}(k)} \varphi_{x;s}(\ell)|.
 \]
 The first equation follows by applying the Cauchy-Schwarz
 inequality to the inner sum. The second equation
 follows from the first by integrating and then applying the
 Cauchy-Schwarz to the terms of the form
 \[
  \int \left( \sum_{n_{x;s} = n} |\ol{\varphi_{x;s}(k)}|^2 \right)^{\frac{1}{2}}
   \left( \sum_{n_{x;s} = n} |\varphi_{x;s}(\ell)|^2\right)^{\frac{1}{2}} d\mu(x).
 \]
\end{proof}

\begin{corollary}\label{cor:A}
For $C, \gamma > 0$ assume for $n, \ell \in \Z^d,$ that 
 \be
  \int \sum_{n_{x;s} = n} |\varphi_{x;s}(\ell)|^2 d\mu(x) \leq C\E^{-2 \gamma |n-\ell|}.
 \ee
 Then
 \be
  \int |\spr{\delta_k}{\E^{-\I t H_x} \delta_{\ell}}| d\mu(x)
   \leq \begin{cases}C_1(\gamma,d,C)(1+|k-\ell|^{d-1})  \E^{-\gamma
     |k-\ell|}, & d>1,\\
C \left(\frac{1 + \gamma}{\gamma}  + |k-\ell|\right)
    \E^{-\gamma |k-\ell|}, & d=1.\end{cases}
 \ee
\end{corollary}

\section{Almost localization}
\label{sec:al}

The goal of this section is to discuss consequences of the work
\cite{aj10} and to give the proof of Theorem~\ref{thm:main}.

\begin{definition}
 Let $\theta\in\R$, $k \in \Z$, and $\eta > 0$.
 Then $k$ is called $\eta$-resonant for $\theta$ if 
 \be
  \|2 \theta - k \alpha\| \leq \E^{-\eta |k|}.
 \ee
 Given $\theta$ and $\eta$, we denote by $k_j$ the set
 of $\eta$-resonances for $\theta$.
\end{definition}

Recall furthermore, that the Almost--Mathieu equation
is given by
\be
 h_{\lambda,\alpha,\theta} u(n) = u(n+1) + u(n-1) + 2 \lambda \cos(2\pi (n \alpha + \theta)) u(n).
\ee
We use the notation $h_{\lambda,\alpha,\theta}$ in contrast with that
of (\ref{am}) to emphasize that
this is meant as a difference equation.
Theorem~5.1. in \cite{aj10} states that

\begin{theorem}
For $\lambda, C_0 > 1$ there exists $\eta(\lambda)>0$ and
$\beta(\lambda, C_0)>0$ so that for $\alpha$ with $\beta(\alpha)<\beta$ there exist
 $C_1, \gamma> 0$, such that for $\theta\in\R$
 and any solution
 $u$ of $h_{\lambda,\alpha,\theta} u = E u$ with $u(0) = 1$ 
 and $|u(n)| \leq 1+ |n|$, we have
 \be
  |u(n)| \leq C_1 \E^{-\gamma |n|}
 \ee
 for $C_0 (1 + |k_j|) \leq |n| \leq \frac{1}{C_0} |k_{j+1}|$
 where $k_j$ denote the $\eta$-resonances of $\theta$.
\end{theorem}

This theorem implies

\begin{theorem}
 We have that
 \be
  \int_0^{1} \left(\sum_{n = n_{\theta;s}} |\varphi_{\theta; s}(\ell)|^{2} \right) d\theta
   \leq (C_1)^2 \E^{-2 \gamma |n - \ell|} + \E^{-\frac{\eta}{C_0} |n-\ell|}.
 \ee
\end{theorem}

\begin{proof}
 By replacing $\theta$ by $\theta - n \alpha$ and $\ell$ by $\ell - n$,
 we can assume that $n = 0$. Then for $n_{\theta; s} = 0$ the functions 
 \[
  u_s = \frac{1}{|\varphi_{\theta; s}(0)|}\varphi_{\theta; s}
 \]
 satisfy the assumptions of the previous theorem. So we obtain
 that if $\ell$ is not $\eta$-resonant for $\theta$ that
 \[
  |\varphi_{\theta; s}(\ell)| \leq C_1 |\varphi_{\theta; s}(0)| \E^{-\gamma |\ell|}.
 \]
 By orthogonality $\sum_{s} |\varphi_{\theta; s}(0)|^2 = 1$, thus we obtain for these
 $\theta$ that
 \[
  \sum_{n_{\theta; s} = 0} |\varphi_{\theta; s}(\ell)|^2 \leq (C_1)^2 \E^{-2\gamma |\ell|}.
 \]
 Hence, we obtain that
 \begin{align*}
  \int_0^1&\left(\sum_{n_{\theta;s} = 0} |\varphi_{\theta; s}(\ell)|^{2} \right)d\theta
   \leq (C_1)^2 \E^{- 2\gamma |\ell|} \\
   &+ \left|\left\{\theta: \quad 
    \exists k\text{ $\eta$-resonances of $\theta$ with }
     \frac{1}{C_0} |k| \leq |\ell| \leq C_0 (1 + |k|)\right\}\right|.
 \end{align*}
 For fixed $k$ the set of $\theta$ such that
 $k$ is $\eta$-resonant for $\theta$ has measure $\E^{-\eta |k|}$.
 As the smallest possible choice for $k$ in the previous
 equation is $\frac{1}{C_0} |\ell|$ the claim follows.
\end{proof}

\begin{proof}[Proof of Theorem~\ref{thm:main} and Corollary~\ref{cor:main}]
 By the previous theorem, the assumptions of Corollary~\ref{cor:A} hold.
 The claim follows.
\end{proof}

\end{document}